\pdfoutput=1
\documentclass[a4paper,12pt]{amsart}
\usepackage{amssymb}
\usepackage{amsfonts}
\usepackage{latexsym}
\title{The $q$-multiple gamma functions of Barnes-Milnor type}
\author{Hanamichi Kawamura}
\begin{document}
\newtheorem{Thm}{Theorem}[section]
\newtheorem{Def}[Thm]{Definition}
\newtheorem{Lem}[Thm]{Lemma}
\newtheorem{Prop}[Thm]{Proposition}
\newtheorem{Cor}[Thm]{Corollary}
\begin{abstract}
The multiple gamma functions of BM (Barnes-Milnor) type and the $q$-multiple gamma functions have been studied independently. In this paper, we introduce a new generalization of the multiple gamma functions called the $q$-BM multiple gamma function including those functions and prove some properties the BM multiple gamma functions satisfy for them. 
\end{abstract}
\maketitle
\section{Introduction}
The $q$-gamma function has been studied since Barnes \cite{ja}. Jackson defined it as the reciprocal of the $q$-Pochhammer symbol, which satisfies the periodicity and some transformation properties like the usual gamma functions. Almost 90 years later, Kurokawa \cite{kogi} introduced the $q$-multiple gamma and sine function. This generalization is also derived from the multiple gamma functions defined by Barnes \cite{barnes}. In \cite{kogi}, some properties, such as the multiplication formula and the periodicity, were proved. A bit later, Tanaka \cite{tanaka} pointed out that the $q$-multiple gamma functions can be regarded as Appell's O-function.

Let $w,\omega_1,\omega_2,\cdots,\omega_r$ be complex numbers with positive real parts. Barnes' multiple gamma functions are defined by
\begin{eqnarray*}\displaystyle\Gamma_r(w;{\boldsymbol{\omega}})=\exp\left(\left.\frac{\partial}{\partial s}\zeta_r(s,w;{\boldsymbol{\omega}})\right|_{s=0}\right),\end{eqnarray*}
where $\zeta_r$ are the multiple Hurwitz' zeta functions
\begin{eqnarray*}\displaystyle\zeta_r(s,w;{\boldsymbol{\omega}})=\sum_{{\bf n}\geq{\bf 0}} ({\bf n}\cdot{\boldsymbol{\omega}}+w)^{-s}\qquad{(\mathrm{Re(s)>r})},\end{eqnarray*}
${\bf 0}=(0,\cdots,0)$, ${\bf n}=(n_1,\cdots,n_r)$, ${\boldsymbol{\omega}}=(\omega_1,\cdots,\omega_r)$, $\boldsymbol{n}\geq\boldsymbol{0}\overset{\mathrm{def}}{\Leftrightarrow}n_i\geq{0}\,(i=1,\cdots,r)$ and ${\bf n}\cdot{\boldsymbol{\omega}}=n_1\omega_1+\cdots+n_r\omega_r$. We use similar notation for other vectors. Kurokawa's definition of the $q$-multiple gamma functions is
\begin{eqnarray*}\displaystyle\Gamma_r^q(w;{\boldsymbol{\omega}})=\frac{\Gamma_{r+1}(w;({\boldsymbol{\omega}},\tau'))\Gamma_{r+1}(w;({\boldsymbol{\omega}},-\tau'))}{\Gamma_r(w;{\boldsymbol{\omega}})},\end{eqnarray*}
where $q=\exp(-2\pi i/\tau')$.

In 2007, Kurokawa and Ochiai \cite{ki} constructed the theory of the multiple gamma functions of BM (Barnes-Milnor) type:
\begin{eqnarray*}\displaystyle\Gamma_{r,k}(w;{\boldsymbol{\omega}})=\exp\left(\left.\frac{\partial}{\partial s}\zeta_r(s,w;{\boldsymbol{\omega}})\right|_{s=-k}\right).\end{eqnarray*}

The case of $k=0$ is the Barnes' multiple gamma functions $\Gamma_r$. The parameters $r$ and $k$ are usually called ``order" and ``depth" respectively. Their generalization of Kinkelin's formula enables the function $\Gamma_{r,k}$ to bring down its depth by integration:
\\
\begin{Thm}[\cite{ki}, Theorem 4]
For $k\geq{1}$, we have
\begin{eqnarray*}\displaystyle\int_0^w \left(\log\Gamma_{r,k-1}(t;{\boldsymbol{\omega}})-\frac{1}{k}\zeta_r(1-k,t;{\boldsymbol{\omega}})\right)\,dt=\frac{1}{k}\log\frac{\Gamma_{r,k}(w;{\boldsymbol{\omega}})}{\Gamma_{r,k}(0;{\boldsymbol{\omega}})}.\end{eqnarray*}
\end{Thm}
Moreover, Kurokawa and Wakayama \cite{pd} investigated period deformations and a generalization of Raabe's integral formula for the BM multiple gamma functions. These results can be written as
\begin{eqnarray*}\displaystyle\lim_{\mu_j\rightarrow{\infty}\atop{1\leq{j}\leq{l}}} \frac{|{\boldsymbol{\alpha}}|_{\times}}{|{\boldsymbol{\mu}}|_{\times}}\log\Gamma_{r+l,k}\left(w;\left({\boldsymbol{\omega}},\frac{{\boldsymbol{\alpha}}}{{\boldsymbol{\mu}}}\right)\right)\qquad\qquad\\=\frac{(-1)^lk!}{(l+k)!}\biggl(\log\Gamma_{r,l+k}(w;{\boldsymbol{\omega}})\biggr.\qquad\qquad\qquad\qquad\qquad\end{eqnarray*}$$\displaystyle+\left.\left(\frac{1}{k+1}+\cdots+\frac{1}{k+l}\right)\zeta_r(-k-l,w;{\boldsymbol{\omega}})\right),$$
where $|{\boldsymbol{\omega}}|_{\times}=\prod_{i=1}^r \omega_i$. This theorem allows adjusting the order and the depth of $\Gamma^q_{r,k}$ by removing some parts of periods. Our purpose in this paper is finding the $q$-analogue of the above theorems and its applications. To prove them, we define the $q$-BM multiple gamma functions as follows:
\begin{eqnarray*}\displaystyle\Gamma_{r,k}^q(w;{\boldsymbol{\omega}})=\frac{\Gamma_{r+1,k}(w;({\boldsymbol{\omega}},\tau'))\Gamma_{r+1,k}(w;({\boldsymbol{\omega}},-\tau'))}{\Gamma_{r,k}(w;{\boldsymbol{\omega}})}.\end{eqnarray*}

At the first onset, we generalize Shibukawa's expression (\cite{shib}, Corollary 4.9):
\begin{Thm}
Let $\mathrm{Li}_s(x)=\sum_{n\geq{1}} x^n/n^s$ be the usual polylogarithm function. For $k\geq{0}$, we have
\begin{eqnarray*}\displaystyle\Gamma_{r,k}^q(w;{\boldsymbol{\omega}})=\exp\left(\frac{k!}{(\log q)^k}\sum_{{\bf n}\geq{\bf 0}}\mathrm{Li}_{k+1}(q^{{\bf n}\cdot{\boldsymbol{\omega}}+w})\right).\end{eqnarray*}
\end{Thm}

Additionally, we show that the $q$-multiple gamma functions satisfies the result of Kurokawa and Ochiai more concisely than the multiple gamma functions.
\begin{Thm}\label{ko}
We have
\begin{eqnarray*}\displaystyle\lim_{\mu_j\rightarrow{\infty}\atop{1\leq{j}\leq{l}}} \Gamma_{r+l,k}^q\left(w;\left({\boldsymbol{\omega}},\frac{{\boldsymbol{\alpha}}}{{\boldsymbol{\mu}}}\right)\right)^{|{\boldsymbol{\alpha}}|_{\times}/|{\boldsymbol{\mu}}|_{\times}}=\Gamma_{r,l+k}^q(w;{\boldsymbol{\omega}})^{\frac{(-1)^lk!}{(l+k)!}},\end{eqnarray*}
where $\displaystyle\frac{{\boldsymbol{\alpha}}}{{\boldsymbol{\mu}}}=\left(\frac{\alpha_1}{\mu_1},\cdots,\frac{\alpha_l}{\mu_l}\right)$.
\end{Thm}
\section{Product expression of $\Gamma_{r,k}^q$}
In this section, we prove Theorem 1.2.
\begin{proof}
Let
\begin{eqnarray*}\displaystyle\zeta_r^q(s,w;{\boldsymbol{\omega}})=\zeta_{r+1}(s,w;({\boldsymbol{\omega}},\tau'))+\zeta_{r+1}(s,w;({\boldsymbol{\omega}},-\tau'))\\-\zeta_r(s,w;{\boldsymbol{\omega}}).\end{eqnarray*}

Then we show that
\begin{eqnarray*}\displaystyle\zeta_r^q(s,w;{\boldsymbol{\omega}})&=&\sum_{{\bf n}\geq{\bf 0}} \sum_{n\in\mathbb{Z}} ({\bf n}\cdot{\boldsymbol{\omega}}-n\tau'+w)^{-s}\\&=&(-\tau')^{-s}\sum_{{\bf n}\geq{\bf 0}}\sum_{n\in\mathbb{Z}} \left(n-\frac{1}{\tau'}({\bf n}\cdot{\boldsymbol{\omega}}+w)\right)^{-s}\\&=&\frac{(-\log q)^s}{\Gamma(s)}\sum_{{\bf n}\geq{{\bf 0}}}\sum_{n=1}^{\infty} n^{s-1}q^{n({\bf n}\cdot{\boldsymbol{\omega}}+w)},\end{eqnarray*}
where we used Lipschitz' formula
\begin{eqnarray*}\displaystyle\sum_{n\in\mathbb{Z}} (n+z)^{-s}=\frac{(-2\pi i)^s}{\Gamma(s)}\sum_{n=1}^{\infty} n^{s-1}e^{2\pi inz}.\end{eqnarray*}

Hence we get
\begin{eqnarray*}\displaystyle\Gamma^q_{r,k}(w;{\boldsymbol{\omega}})&=&\exp\biggl(\frac{\partial}{\partial s}\frac{(-\log q)^s}{\Gamma(s)}\sum_{{\bf n}\geq{{\bf 0}}}\sum_{n=1}^{\infty} n^{s-1}\biggr.\\&{}&\biggl.\times\biggl.q^{n({\bf n}\cdot{\boldsymbol{\omega}}+w)}\biggr|_{s=-k}\biggr)\\&=&\exp\biggl(\left.\left(\frac{\partial}{\partial s}\frac{(-\log q)^s}{\Gamma(s)}\right)\right|_{s=-k}\biggr.\\&{}&\times\biggl.\sum_{{\bf n}\geq{{\bf 0}}}\sum_{n=1}^{\infty} n^{-k-1}q^{n({\bf n}\cdot{\boldsymbol{\omega}}+w)}\biggr),\end{eqnarray*}
and we can easily get
\begin{eqnarray*}\displaystyle\lim_{s\rightarrow{-k}}\frac{d}{ds} \frac{a^s}{\Gamma(s)}=\frac{k!}{(-a)^k}.\end{eqnarray*}
\end{proof}
\begin{Cor}[\cite{tanaka}, Theorem 1]
We have
\begin{eqnarray*}\displaystyle\Gamma_r^q(w;{\boldsymbol{\omega}})=\prod_{{\bf n}\geq{\bf 0}}(1-q^{{\bf n}\cdot{\boldsymbol{\omega}}+w})^{-1}.\end{eqnarray*}
\end{Cor}
\section{Period deformation and Raabe's formula}
We prove period deformation and Raabe's formula of $\Gamma^q_{r,k}$ in the similar form of \cite{pd}.
\begin{Prop}[Period deformation]
We have
$$\displaystyle\underbrace{\int_0^1\cdots\int_0^1}_{l} \log\Gamma^q_{r+l,k}(w+{\bf t}\cdot{\boldsymbol{\alpha}};({\boldsymbol{\omega}},{\boldsymbol{\alpha}}))\,d{\bf t}$$$$=\lim_{\mu_j\rightarrow{\infty}\atop{1\leq{j}\leq{l}}}\frac{1}{|{\boldsymbol{\mu}}|_{\times}}\log\Gamma_{r+l,k}\left(w;\left({\boldsymbol{\omega}},\frac{{\boldsymbol{\alpha}}}{{\boldsymbol{\mu}}}\right)\right).$$
\end{Prop}
\begin{proof}
We obtain 
$$\displaystyle\int_0^1\cdots\int_0^1 \zeta^q_{r+l,k}(w+{\bf t}\cdot{\boldsymbol{\alpha}};({\boldsymbol{\omega}},{\boldsymbol{\alpha}}))\,d{\bf t}$$\begin{eqnarray*}\displaystyle&=&\lim_{\mu_j\rightarrow{\infty}\atop{1\leq{j}\leq{l}}}\frac{1}{|{\boldsymbol{\mu}}|_{\times}}\sum_{k_j=0\atop{1\leq{j}\leq{l}}}^{\mu_j-1}\zeta^q_{r+l,k}\left(w+\frac{{\bf k}\cdot{\boldsymbol{\alpha}}}{{\boldsymbol{\mu}}};({\boldsymbol{\omega}},{\boldsymbol{\alpha}})\right)\\&=&\lim_{\mu_j\rightarrow{\infty}\atop{1\leq{j}\leq{l}}}\frac{1}{|{\boldsymbol{\mu}}|_{\times}}\sum_{k_j=0\atop{1\leq{j}\leq{l}}}^{\mu_j-1} \sum_{{\bf n}\geq{\bf 0}}\sum_{{\bf m}\geq{\bf 0}}\sum_{n\in\mathbb{Z}} \biggl({\bf n}\cdot{\boldsymbol{\omega}}+{\bf m}\cdot{\boldsymbol{\alpha}}\biggr.\end{eqnarray*}$$\biggl.-n\tau'+w+\frac{{\bf k}\cdot{\boldsymbol{\alpha}}}{{\boldsymbol{\mu}}}\biggr)^{-s}.$$
Here, by substituting $\mu_jm+k_j=\nu_j$, we have
$$\int_0^1\cdots\int_0^1 \zeta^q_{r+l,k}(w+{\bf t}\cdot{\boldsymbol{\alpha}};({\boldsymbol{\omega}},{\boldsymbol{\alpha}}))\,d{\bf t}$$\begin{eqnarray*}\displaystyle&=&\lim_{\mu_j\rightarrow{\infty}\atop{1\leq{j}\leq{l}}}\frac{1}{|{\boldsymbol{\mu}}|_{\times}}\sum_{{\bf n}\geq{\bf 0}}\sum_{{\bf m}\geq{\bf 0}}\sum_{n\in\mathbb{Z}} \left({\bf n}\cdot{\boldsymbol{\omega}}+\frac{{\boldsymbol{\nu}}\cdot{\boldsymbol{\alpha}}}{{\boldsymbol{\mu}}}\right.\biggl.-n\tau'+w\biggr)^{-s}\\&=&\lim_{\mu_j\rightarrow{\infty}\atop{1\leq{j}\leq{l}}} \frac{1}{|{\boldsymbol{\mu}}|_{\times}}\zeta_{r+l}\left(s,w;\left({\boldsymbol{\omega}},\frac{{\boldsymbol{\alpha}}}{{\boldsymbol{\mu}}}\right) \right).\end{eqnarray*}
Thus the statement is given by the differentiation at $s=-k$.
\end{proof}
\begin{Thm}[Raabe's formula]\label{qgrf}
We have
\begin{eqnarray*}\displaystyle\underbrace{\int_0^1\cdots\int_0^1}_{l} \log\Gamma^q_{r+l,k}(w+{\bf t}\cdot{\boldsymbol{\alpha}};({\boldsymbol{\omega}},{\boldsymbol{\alpha}}))\,d{\bf t}\\=\frac{(-1)^lk!}{|{\boldsymbol{\alpha}}|_{\times}(l+k)!}\biggl(\log\Gamma^q_{r,l+k}(w;{\boldsymbol{\omega}})\biggr.\qquad\qquad\qquad\\+\left.\left(\frac{1}{k+1}+\cdots+\frac{1}{k+l}\right)\zeta^q_r(-k-l,w;{\boldsymbol{\omega}})\right).\end{eqnarray*}
\end{Thm}
\begin{proof}
We plan to prove is similar to that of \cite{pd}, Theorem 4. We only have to show 
\begin{eqnarray*}\displaystyle\int_0^{\alpha_1}\cdots\int_0^{\alpha_l} \zeta^q_{r+l}(s,w+|{\bf t}|;({\boldsymbol{\omega}},{\boldsymbol{\alpha}}))\,d{\bf t}=\frac{\zeta^q_r(s-l,w;{\boldsymbol{\omega}})}{\prod_{i=1}^l (s-i)}\end{eqnarray*}
when $\mathrm{Re}(s)>-r-l$ because both sides are meromorphically extendable on the whole plane $\mathbb{C}$. Moreover, it is sufficient that we get
\begin{eqnarray*}\displaystyle\int_0^{\alpha_l} \zeta^q_{r+l}(s,w+|{\bf t}|;({\boldsymbol{\omega}},{\boldsymbol{\alpha}}))\,dt_l=\frac{\zeta^q_{r+l-1}(s-1,w+|{\bf t}\langle{l}\rangle|;({\boldsymbol{\omega}};{\boldsymbol{\alpha}}\langle{l}\rangle))}{s-1},\end{eqnarray*}
where ${\boldsymbol{\alpha}}\langle{l}\rangle$ means $(\alpha_1,\cdots,\alpha_{l-1})$. Thus we show this. By substitution $w+|{\bf t}|\mapsto{z}$, we get
$$\displaystyle\int_0^{\alpha_l} \zeta^q_{r+l}(s,w+|{\bf t}|;({\boldsymbol{\omega}},{\boldsymbol{\alpha}}))\,dt_l$$\begin{eqnarray*}\displaystyle&=&\int_{w+|{\bf t}\langle{l}\rangle|}^{w+|{\bf t}\langle{l}\rangle|+\alpha_l} \zeta^q_{r+l}(s,z;({\boldsymbol{\omega}},{\boldsymbol{\alpha}}))\,dz\\&=&\int_{w+|{\bf t}\langle{l}\rangle|}^{\infty} \zeta^q_{r+l}(s,z;({\boldsymbol{\omega}},{\boldsymbol{\alpha}}))\,dz-\int_{w+|{\bf t}\langle{l}\rangle|+\alpha_l}^{\infty} \zeta^q_{r+l}(s,z;({\boldsymbol{\omega}},{\boldsymbol{\alpha}}))\,dz\\&=&\int_{w+|{\bf t}\langle{l}\rangle|}^{\infty} \zeta^q_{r+l-1}(s,z;({\boldsymbol{\omega}},{\boldsymbol{\alpha}}\langle{l}\rangle))\,dz.\end{eqnarray*}
where we used the ladder structure of the $q$-multiple Hurwitz' zeta functions
\begin{eqnarray*}\displaystyle\zeta^q_r(s,w+\omega_i;{\boldsymbol{\omega}})=\zeta^q_r(s,w;{\boldsymbol{\omega}})-\zeta^q_{r-1}(s,w;{\boldsymbol{\omega}}\langle{i}\rangle)\end{eqnarray*}
which was proved in \cite{tanaka}, Proposition 2.1. Therefore we obtain
$$\int_0^{\alpha_l} \zeta^q_{r+l}(s,w+|{\bf t}|;({\boldsymbol{\omega}},{\boldsymbol{\alpha}}))\,dt_l$$\begin{eqnarray*}\displaystyle&=&\sum_{{\bf n}\geq{\bf 0}}\sum_{{\bf m}\langle{l}\rangle\geq{\bf 0}}\sum_{n\in\mathbb{Z}}\int_{w+|{\bf t}\langle{l}\rangle|}^{\infty} ({\bf n}\cdot{\boldsymbol{\omega}}+{\bf m}\langle{l}\rangle\cdot{\boldsymbol{\alpha}}\langle{l}\rangle\\&{}&\qquad\qquad-n\tau'+w)^{-s}\\&=&\frac{\zeta^q_{r+l-1}(s-1,w+|{\bf t}\langle{l}\rangle|;({\boldsymbol{\omega}};{\boldsymbol{\alpha}}\langle{l}\rangle))}{s-1}.\end{eqnarray*}

From the above, we get the desired result inductively.
\end{proof}
\section{Proof of Theorem 1.3 and corollaries}
In this section, we prove Theorem 1.3 and show its applications.
\begin{proof}
It follows that
$$\lim_{\mu_j\rightarrow{\infty}\atop{1\leq{j}\leq{l}}}\frac{1}{|{\boldsymbol{\mu}}|_{\times}}\log\Gamma^q_{r+l,k}\left(w;\left({\boldsymbol{\omega}},\frac{{\boldsymbol{\alpha}}}{{\boldsymbol{\mu}}}\right)\right)$$\begin{eqnarray*}\displaystyle&=&\lim_{\mu_j\rightarrow{\infty}\atop{1\leq{j}\leq{l}}}\frac{1}{|{\boldsymbol{\mu}}|_{\times}}\frac{k!}{(\log q)^{k}}\\&{}&\times\sum_{n=1}^{\infty}\frac{q^{nw}}{n^{k+1}\prod_{i=1}^r (1-q^{\omega_in})\prod_{j=1}^l (1-q^{\alpha_jn/\mu_j})}\\&=&\frac{(-1)^lk!}{|{\boldsymbol{\alpha}}|_{\times}(\log q)^{l+k}}\sum_{n=1}^{\infty}\frac{q^{nw}}{n^{k+l+1}\prod_{i=1}^r(1-q^{\omega_in})},\end{eqnarray*}
since
\begin{eqnarray*}\displaystyle\lim_{\mu_j\rightarrow{\infty}} \mu_j\left(1-q^{\frac{\alpha_jn}{\mu_j}}\right)&=&\lim_{\mu_j\rightarrow{\infty}} \mu_j\left(1-\sum_{\nu=0}^{\infty} \frac{(\frac{\alpha_jn\log q}{\mu_j})^{\nu}}{\nu!}\right)\\&=&-\alpha_jn\log q.\end{eqnarray*}
Therefore we obtain the statement.
\end{proof}

This theorem has some applications. In the following corollaries, we see two of them. The first is a proof of the transformation property (or the modular property) of Dedekind's eta function:
\begin{eqnarray*}\displaystyle\eta\left(-\frac{1}{\tau}\right)=\sqrt{\frac{\tau}{i}}\eta(\tau).\end{eqnarray*}
This can be derived from the following corollary which is found in \cite{tanaka}.
\begin{Cor}[\cite{tanaka}, Theorem 2]
We have
\begin{eqnarray*}\displaystyle\rho^q_r({\boldsymbol{\omega}})=-\log q\prod_{{\bf n}\geq{\bf 0}\atop{{\bf n}\neq{{\bf 0}}}} (1-q^{{\bf n}\cdot{\boldsymbol{\omega}}}).\end{eqnarray*}
\end{Cor}
This is the generalization of Shintani's result \cite{shin}, Proposition 2 (1).\\

The second is the vanishing property of $\zeta^q_r(s,w;{\boldsymbol{\omega}})$. By comparing Theorem \ref{ko} and Theorem \ref{qgrf}, we get the following:
\begin{Cor}[\cite{shib}, Corollary 4.8 (2)]
For $n\geq{0}$, we have
\begin{eqnarray*}\displaystyle\zeta_r^q(-n,w;{\boldsymbol{\omega}})=0.\end{eqnarray*}
\end{Cor}
\begin{proof}
It is easy to see that
\begin{eqnarray*}\displaystyle\frac{(-1)^lk!}{|{\boldsymbol{\alpha}}|_{\times}(k+l)!}\left(\frac{1}{k+1}+\cdots+\frac{1}{k+l}\right)\neq{0}.\end{eqnarray*}
\end{proof}

\end{document}